\documentclass{article}%
\usepackage{graphicx}
\usepackage{amsmath}
\usepackage{amsfonts}
\usepackage{amssymb}
\usepackage{setspace}
\usepackage[hmargin=3.5cm,vmargin=3.5cm]{geometry}%
\providecommand{\U}[1]{\protect\rule{.1in}{.1in}}
\providecommand{\U}[1]{\protect\rule{.1in}{.1in}}
\providecommand{\U}[1]{\protect\rule{.1in}{.1in}}
\providecommand{\U}[1]{\protect\rule{.1in}{.1in}}
\providecommand{\U}[1]{\protect\rule{.1in}{.1in}}
\providecommand{\U}[1]{\protect\rule{.1in}{.1in}}
\providecommand{\U}[1]{\protect\rule{.1in}{.1in}}
\providecommand{\U}[1]{\protect\rule{.1in}{.1in}}
\providecommand{\U}[1]{\protect\rule{.1in}{.1in}}
\providecommand{\U}[1]{\protect\rule{.1in}{.1in}}
\providecommand{\U}[1]{\protect\rule{.1in}{.1in}}
\providecommand{\U}[1]{\protect\rule{.1in}{.1in}}
\providecommand{\U}[1]{\protect\rule{.1in}{.1in}}
\providecommand{\U}[1]{\protect\rule{.1in}{.1in}}
\providecommand{\U}[1]{\protect\rule{.1in}{.1in}}
\providecommand{\U}[1]{\protect\rule{.1in}{.1in}}
\providecommand{\U}[1]{\protect\rule{.1in}{.1in}}
\providecommand{\U}[1]{\protect\rule{.1in}{.1in}}
\providecommand{\U}[1]{\protect\rule{.1in}{.1in}}
\providecommand{\U}[1]{\protect\rule{.1in}{.1in}}
\providecommand{\U}[1]{\protect\rule{.1in}{.1in}}
\providecommand{\U}[1]{\protect\rule{.1in}{.1in}}
\providecommand{\U}[1]{\protect\rule{.1in}{.1in}}
\providecommand{\U}[1]{\protect\rule{.1in}{.1in}}
\providecommand{\U}[1]{\protect\rule{.1in}{.1in}}
\providecommand{\U}[1]{\protect\rule{.1in}{.1in}}
\providecommand{\U}[1]{\protect\rule{.1in}{.1in}}
\providecommand{\U}[1]{\protect\rule{.1in}{.1in}}
\providecommand{\U}[1]{\protect\rule{.1in}{.1in}}
\providecommand{\U}[1]{\protect\rule{.1in}{.1in}}
\providecommand{\U}[1]{\protect\rule{.1in}{.1in}}
\providecommand{\U}[1]{\protect\rule{.1in}{.1in}}
\providecommand{\U}[1]{\protect\rule{.1in}{.1in}}
\providecommand{\U}[1]{\protect\rule{.1in}{.1in}}
\providecommand{\U}[1]{\protect\rule{.1in}{.1in}}
\providecommand{\U}[1]{\protect\rule{.1in}{.1in}}
\providecommand{\U}[1]{\protect\rule{.1in}{.1in}}
\providecommand{\U}[1]{\protect\rule{.1in}{.1in}}

\newtheorem{theorem}{Theorem}
{}

\newtheorem{corollary}{Corollary}

{}

\newtheorem{remark}{Remark}

\newenvironment{proof}[1][Proof]{\textbf{#1.} }{\ \rule{0.5em}{0.5em}}

\begin{document}

\title{Isospectral \ Mathieu-Hill Operators }
\author{O. A. Veliev\\{\small Depart. of Math., Dogus University, Ac\i badem, Kadik\"{o}y, \ }\\{\small Istanbul, Turkey.}\ {\small e-mail: oveliev@dogus.edu.tr}}
\date{}
\maketitle

\begin{abstract}
In this paper we prove that the spectra of the \ Mathieu-Hill Operators with
potentials $ae^{-i2\pi x}+be^{i2\pi x}$ and $ce^{-i2\pi x}+de^{i2\pi x}$ are
the same if and only if $ab=cd,$ where $a,b,c$ and $d$ are complex numbers.
This implies some corollaries about the extension of Harrell-Avron-Simon
formula. Moreover, we find explicit formulas for \ the eigenvalues and
eigenfunctions of the $t$-periodic boundary value problem for the Hill
operator with Gasymov's potential.



\end{abstract}

Let $H(a,b)$ be the Hill operator generated in $L_{2}(-\infty,\infty)$ by the expression%

\begin{equation}
-y^{^{\prime\prime}}(x)+q(x)y(x)
\end{equation}
with potential
\begin{equation}
q(x)=ae^{-i2\pi x}+be^{i2\pi x},
\end{equation}
where $a$ and $b$ are complex numbers. It is well-known that (see [4, 8]) the
spectrum $S(H(a,b))$ of the operator $H(a,b)$ is the union of the spectra
$S(H_{t}(a,b))$ of the operators $H_{t}(a,b)$ for $t\in(-\pi,\pi],$ where
$H_{t}(a,b)$ is the operator generated in $L_{2}[0,1]$ by (1) with potential
(2) and by the boundary conditions
\begin{equation}
y(1)=e^{it}y(0),\text{ }y^{^{\prime}}(1)=e^{it}y^{^{\prime}}(0).
\end{equation}

First we prove that if $ab=cd$, then%
\begin{equation}
S(H(a,b))=S(H(c,d)),\text{ }S(H_{t}(a,b))=S(H_{t}(c,d))
\end{equation}
for all $t\in(-\pi,\pi].$ For this we obtain the asymptotic formulas, uniform
with respect to

$t\in\lbrack\rho,\pi-\rho],$ for eigenvalues and eigenfunctions of the
operators $H_{t},$ where $\rho$ is a fixed number from the interval
$(0,\frac{\pi}{2}).$ Note that the formula\ $f(k,t)=O(h(k))$ is said to be
uniform with respect to $t$ in a set $I$ if there exist positive constants $M$
and $N,$ independent of $t,$ such that $\mid f(k,t))\mid<M\mid h(k)\mid$ for
all $t\in I$ and $\mid k\mid\geq N.$

To obtain the uniform asymptotic formulas for eigenvalues $\lambda_{n}(t)$ and
corresponding normalized eigenfunctions $\Psi_{n,t}(x)$ for $t\in\lbrack
\rho,\pi-\rho]$, as $n\rightarrow\infty$ we use the formulas
\begin{equation}
(\lambda_{n}(t)-(2\pi n+t)^{2})(\Psi_{n,t},e^{i(2\pi n+t)x})=(q\Psi
_{n,t},e^{i(2\pi n+t)x})
\end{equation}
and
\begin{equation}
(\lambda_{n}(t)-(2\pi(n-k)+t)^{2})(\Psi_{n,t},e^{i(2\pi(n-k)+t)x}%
)=(q\Psi_{n,t},e^{i(2\pi(n-k)+t)x}),
\end{equation}
where $(.,.)$ is the inner product in $L_{2}[0,1]$. Formulas (5) and (6) can
be obtained from
\begin{equation}
-\Psi_{n,t}^{^{\prime\prime}}(x)+q(x)\Psi_{n,t}=\lambda_{n}(t)\Psi_{n,t}(x)
\end{equation}
by multiplying $e^{i(2\pi n+t)x}$ and $e^{i(2\pi(n-k)+t)x}$ respectively.

The uniform asymptotic formulas for the operator $L_{t}(q),$with $q\in
L_{1}[0,1],$ generated in $L_{2}[0,1]$ by (1) and (3) is obtained in [9],
where we proved the following:

\textit{The large eigenvalue }$\lambda_{n}(t)$\textit{ \ and the corresponding
eigenfunction }$\Psi_{n,t}(x)$\textit{ of the operator }$L_{t}(q)$\textit{ for
}$t\neq0,\pi,$\textit{ satisfy the following asymptotic formulas }%
\begin{equation}
\lambda_{n}(t)=(2\pi n+t)^{2}+O(\frac{ln\left\vert n\right\vert }{n}),\text{
}\Psi_{n,t}(x)=e^{i(2\pi n+t)x}+O(\frac{1}{n}).
\end{equation}
\textit{These asymptotic formulas are uniform with respect to }$t$\textit{ in
}$[\rho,\pi-\rho],$ \textit{where }$\rho$\textit{ is a fixed number from
}$(0,\frac{\pi}{2}).$\textit{ There exists a positive number }$N(\rho
),$\textit{ independent of }$t,$\textit{ such that the eigenvalues }%
$\lambda_{n}(t)$\textit{ for \ }$t\in\lbrack\rho,\pi-\rho]$\textit{ and}$\mid
n\mid>N(\rho)$\textit{ are simple.}

In [9], we obtained (8) by iteration of the formula (5). However, for the
convenience of the readers and taking into account that we need to consider
the terms of the asymptotic formulas in detail, we repeat the iteration here.
Using (2) in (5) we get
\begin{equation}
(\lambda_{n}(t)-(2\pi n+t)^{2})(\Psi_{n,t}(x),e^{i(2\pi n+t)x})=\sum_{n_{1}%
}q_{n_{1}}(\Psi_{n,t}(x),e^{i(2\pi(n-n_{1})+t)x}),
\end{equation}
where
\begin{equation}
q_{n}=(q(x),e^{i2\pi nx}),\text{ }q_{-1}=a,\text{ }q_{1}=b,\text{ }%
q_{n}=0,\text{ }\forall n\neq\pm1.
\end{equation}
In (6) replacing $k$ by $n_{1}$ and then using (2) we get
\begin{equation}
(\lambda_{n}(t)-(2\pi(n-n_{1})+t)^{2})(\Psi_{n,t},e^{i(2\pi(n-n_{1}%
)+t)x})=\sum_{n_{2}}q_{n_{2}}(\Psi_{n,t},e^{i(2\pi(n-n_{1}-n_{2})+t)x}).
\end{equation}
Let
\[
U(n,t)=\{\lambda\in\mathbb{C}:\left\vert \lambda-(2\pi n+t)^{2}\right\vert
\leq1\},\text{ }n\in\mathbb{Z}\text{.}%
\]
By (8) the disk $U(n,t)$ for $t\in\lbrack\rho,\pi-\rho]$ and $\left\vert
n\right\vert >N(\rho)$ contains only one simple eigenvalue denoted by
$\lambda_{n}(t).$ Moreover, it is well known that [4] for $t=0$ and
$\left\vert n\right\vert \gg1$ the disk $U(n,t)$ contains two eigenvalues,
denoted here by $\lambda_{n}(0)$ and $\lambda_{-n}(0)$. One can readily see
that if $t\in\lbrack\rho,\pi-\rho],$ $k\neq0$ and $\left\vert n\right\vert
>N(\rho)\gg1,$ then
\begin{equation}
\left\vert \lambda_{n}(t)-(2\pi(n-k)+t)^{2}\right\vert >\left\vert
n\right\vert \rho,\text{ }\left\vert \lambda-(2\pi(n-k)+t)^{2}\right\vert
>\left\vert n\right\vert \rho.
\end{equation}
for all $t\in\lbrack\rho,\pi-\rho]$ and $\lambda\in U(n,t).$ 

Now we itarate (9) as follows. By (12), the last multiplicand $(\Psi
_{n,t},e^{i(2\pi(n-n_{1})+t)x})$ in (9) can be replaced with the right-hand
side of (11) divided by $\lambda_{n}(t)-(2\pi(n-n_{1})+t)^{2}.$ Doing this
replacement we get
\begin{equation}
(\lambda_{n}(t)-(2\pi n+t)^{2})(\Psi_{n,t},e^{i(2\pi n+t)x})=\sum_{n_{1}%
,n_{2}}\frac{q_{n_{1}}q_{n_{2}}(\Psi_{n,t},e^{i(2\pi(n-n_{1}-n_{2})+t)x}%
)}{\lambda_{n}(t)-(2\pi(n-n_{1})+t)^{2}}.
\end{equation}
Now we isolate the terms in right-hand side of (13) containing the multiplicand

$(\Psi_{n,t},e^{i(2\pi n+t)x})$ which occurs in the case $n_{1}+n_{2}=0$ and
apply the above replacement to the other terms (i.e., case $n_{1}+n_{2}\neq0$
) to get
\[
(\lambda_{n}(t)-(2\pi n+t)^{2})(\Psi_{n,t},e^{i(2\pi n+t)x})=\sum_{n_{1}}%
\frac{q_{n_{1}}q_{-n_{1}}(\Psi_{n,t},e^{i(2\pi n+t)x})}{\lambda_{n}%
(t)-(2\pi(n-n_{1})+t)^{2}}+
\]%
\begin{equation}
\sum_{n_{1},n_{2},n_{3}}\frac{q_{n_{1}}q_{n_{2}}q_{n_{3}}(\Psi_{n,t}%
,e^{i(2\pi(n-n_{1}-n_{2}-n_{3})+t)x})}{(\lambda_{n}(t)-(2\pi(n-n_{1}%
)+t)^{2})(\lambda_{n}(t)-(2\pi(n-n_{1}-n_{2})+t)^{2})}.
\end{equation}
Repeating this process $m-$times (i.e., in the second row of (14), isolating
the terms containing the multiplicand $(\Psi_{n,t},e^{i(2\pi n+t)x})$ applying
the above replacement to the other terms, etc.) we obtain
\begin{equation}
(\lambda_{n}(t)-(2\pi n+t)^{2})(\Psi_{n,t},e^{i(2\pi n+t)x})=A_{m}(\lambda
_{n}(t))(\Psi_{n,t},e^{i(2\pi n+t)x})+R_{m+1}(\lambda_{n}(t)),
\end{equation}
where $A_{m}(\lambda)=\sum_{k=1}^{m}a_{k}(\lambda),$%
\begin{equation}
a_{k}(\lambda)=\sum_{n_{1},n_{2},...,n_{k}}\frac{q_{n_{1}}q_{n_{2}}%
...q_{n_{k}}q_{-n_{1}-n_{2}-...-n_{k}}}{%
{\textstyle\prod\limits_{s=1,2,...,k}}
[\lambda-(2\pi(n-n_{1}-n_{2}-...-n_{s})+t)^{2}]},
\end{equation}%
\begin{equation}
R_{m+1}(\lambda)=\sum_{n_{1},n_{2},...,n_{m+1}}\frac{q_{n_{1}}q_{n_{2}%
}...q_{n_{m}}q_{n_{m+1}}(\Psi_{n,t}(x),e^{i(2\pi(n-n_{1}-n_{2}-...-n_{m+1}%
)+t)x})}{%
{\textstyle\prod\limits_{s=1,2,...,m+1}}
[\lambda-(2\pi(n-n_{1}-n_{2}-...-n_{s})+t)^{2}]}%
\end{equation}
and by (10)%
\begin{equation}
\{n_{1},n_{2},...,n_{k},-n_{1}-n_{2}-...-n_{k}\}\subset\{-1,1\}.
\end{equation}

Let $k$ be an even number: $k=2p.$ Then\ $-n_{1}-n_{2}-...-n_{k}$ is also an
even number, since the numbers $n_{1},n_{2},...,n_{k}$ are $-1$ or $1$ (see
(18)). Therefore by (10), $q_{-n_{1}-n_{2}-...-n_{k}}=0$ and by (16)%
\begin{equation}
a_{2p}(\lambda,t)=0,\text{ }\forall p=1,2,...
\end{equation}

Let $k$ be an odd number: $k=2p-1$. Then the set in the left-hand side of (18)
contains $2p$ numbers that are $-1$ or $1$ and their total sum is $0.$ Hence
$p$ of those numbers are $-1$ and $p$ of those are $1.$ Therefore, by (10) and
(16), $a_{k}(\lambda,t)$ for $k=2p-1$ is the sum of $2^{k}$ terms of the form
\begin{equation}
(ab)^{p}%
{\textstyle\prod\limits_{s=1,2,...,k}}
\left(  \lambda-(2\pi(n-n_{1}-n_{2}-...-n_{s})+t)^{2}\right)  ^{-1}.
\end{equation}
Thus we have
\begin{equation}
\text{ }a_{2p-1}(\lambda,t)=2^{2p-1}(ab)^{p}f_{p}(\lambda,t),\text{ }%
\end{equation}
where $f_{p}(\lambda,t)$ does not depend on $a$ and $b$ and, by (12),
satisfies the following, uniform with respect to $t\in\lbrack\rho,\pi-\rho],$
equality
\begin{equation}
\text{ }f_{p}(\lambda,t)=O((n\rho)^{-2p+1})
\end{equation}
for $\lambda\in U(n,t).$ In the same way we obtain
\begin{equation}
R_{m+1}=O((2M)^{m+1}(n\rho)^{-m-1}),
\end{equation}
where $M=\max\{\mid a\mid,\mid b\mid\}.$ Letting $m$ tend to infinity in (15)
and using (23) we obtain that $\lambda_{n}(t)$ is a root of the equation
\begin{equation}
\lambda-(2\pi n+t)^{2}=A(\lambda,t,ab),
\end{equation}
where, by (19),
\begin{equation}
A(\lambda,t,ab)=\sum_{p=1}^{\infty}a_{2p-1}(\lambda,t).
\end{equation}
It follows from (21) and (22) that $A(\lambda,t,ab),$ for fixed $t,$ is an
analytic function of \textit{ } $\lambda\in U(n,t)$ satisfying the following,
uniform with respect to $t\in\lbrack\rho,\pi-\rho],$ asymptotic formula
\[
A(\lambda,t,ab)=O(n^{-1}).
\]
Therefore the inequality \
\begin{equation}
\mid A(\lambda,t,ab)\mid<\mid(\lambda-(2\pi n+t)^{2}\mid
\end{equation}
holds for all $\lambda$ from the boundary of $U(n,t).$ Since the function
$\lambda-(2\pi n+t)^{2}$ has one root in \ the \ set $U(n,t),$ by the Rouche's
theorem (24) has also one root in the same\ set.\ On the other hand,
$\lambda_{n}(t)$ is a root of (24) lying in $U(n,t).$ Therefore $\lambda\in
U(n,t)$ is an eigenvalue of $H_{t}(a,b)$ if and only if it is a root of (24).

\begin{theorem}
If $ab=cd,$ then (4) holds.
\end{theorem}

\begin{proof}
One can readily see from (21) and (25) that if $ab=cd$ then
\begin{equation}
A(\lambda,t,ab)=A(\lambda,t,cd).
\end{equation}
Let $\mu_{n}(t)$ be the eigenvalue of $H_{t}(c,d)$ lying $U(n,t).$ By (27),
both $\lambda_{n}(t)$ and $\mu_{n}(t)$ are the roots of the same equation (24)
which has unique root in $U(n,t).$ Therefore we have
\begin{equation}
\lambda_{n}(t)=\mu_{n}(t),\text{ }\forall t\in\lbrack\rho,\pi-\rho].
\end{equation}
On the other hand, $\lambda_{n}(t)$ and $\mu_{n}(t)$ are the roots of the
equations
\begin{equation}
F(\lambda)=\cos t,\text{ }G(\lambda)=\cos t,
\end{equation}
where $F(\lambda)$ and $G(\lambda)$ are the Hill's discriminants of the
operators $H_{t}(a,b)$ and $H_{t}(c,d)$ respectively (see [6,9]). Since the
eigenvalue $\lambda_{n}(t)$ for $t\in\lbrack\rho,\pi-\rho]$ and $\mid
n\mid>N(\rho)$ is simple the set $\{\lambda_{n}(t):t\in\lbrack\rho,\pi
-\rho]\}$ is an analytic arc. By (28) and (29) the entire functions
$F(\lambda)$ and $G(\lambda)$ coincide on this arc. Therefore these functions
are identically equal in the complex plane and hence the eigenvalues of
$H_{t}(a,b)$ and $H_{t}(c,d)$ are the roots of the same equation (29) for all
$t\in(-\pi,\pi],$ that is, (4) holds
\end{proof}

\begin{remark}
Note that to prove Theorem 1 we investigated the simplest case $t\in
\lbrack\rho,\pi-\rho]$. In the paper [10] we obtained the uniform asymptotic
formulas in the more complicated case $t\in\lbrack0,\rho]\cup\lbrack\pi
-\rho,\pi].$ In the same way one can prove Theorem1 by using the formulas of
[10]. Indeed, in [10] we proved that (see Theorem 2 and 4 of [10]) \textit{the
eigenvalue }$\lambda_{n}(t)$ for $t\in\lbrack0,\rho]$ and $n>N\gg1,$ is simple
and satisfies the equality
\begin{equation}
(\lambda-(2\pi n+t)^{2}-A(\lambda,t))(\lambda-(2\pi n-t)^{2}-A^{^{\prime}%
}(\lambda,t))=B(\lambda,t)B^{^{\prime}}(\lambda,t),
\end{equation}
where $A(\lambda,t)$ and $A^{^{\prime}}(\lambda,t)$ are defined as (25) and
hence depend only on $ab$ (see (8)-(14) of [10]). The functions $B(\lambda,t)$
and $B^{^{\prime}}(\lambda,t)$ are the sum of $b_{2n+2m-1}(\lambda,t)$ and
$b_{2n+2m-1}^{^{\prime}}(\lambda,t)$ respectively for $m=0,1,2,...$ Moreover,
from (46), and (10) of [10] one can readily see that $2n+m$ of the indices
$n_{1},n_{2},...,n_{2n+2m-1},2n-n_{1}-n_{2}-...-n_{2n+2m-1}$ taking part in
$b_{2n+2m-1}(\lambda,t)$ are $\ 1$ and $m$ indices of them are $-1.$ This
implies that the expression $(b_{2n-1}(\lambda,t))^{-1}b_{2n+2m-1}(\lambda,t)$
depends only on $ab,$ since
\begin{equation}
b_{2n-1}(\lambda,t)=b^{2n}%
{\textstyle\prod\limits_{s=1}^{2n-1}}
\left(  \lambda-(2\pi(n-s)+t)^{2}\right)  ^{-1}.
\end{equation}
Similarly, the expression $(b_{2n-1}^{^{\prime}}(\lambda,t))^{-1}%
b_{2n+2m-1}^{^{\prime}}(\lambda,t)$ depends only on $ab,$ where%
\begin{equation}
b_{2n-1}^{^{\prime}}(\lambda,t)=a^{2n}%
{\textstyle\prod\limits_{s=1}^{2n-1}}
\left(  \lambda-(2\pi(n-s)-t)^{2}\right)  ^{-1}.
\end{equation}
Therefore the right-hand side of (30) also depends only on $ab,$ that is,
\begin{equation}
B(\lambda,t)B^{^{\prime}}(\lambda,t)=f(\lambda,t,ab)
\end{equation}
Using these and arguing as in the proof of Theorem 1 we get the other proof of
Theorem 1.
\end{remark}

Theorem 1 shows that if all the eigenvalues of $H_{t}(a,b)$ for all values of
$t\in(-\pi,\pi]$ are given then one can determine only $ab.$ However, the
following simple theorem shows that one can determine $ab$ by given
subsequence of the eigenvalues of $H_{t}(a,b)$ for some value of $t.$

\begin{theorem}
If for some value of $t\in\lbrack0,\pi]$ and for some sequence $\{n_{k}\}$ the
eigenvalues $\lambda_{n_{k}}(t,a,b)=:\lambda_{n_{k}}(t)$ of $H_{t}(a,b)$ are
given, then one can constructively determine $ab.$
\end{theorem}

\begin{proof}
Let $t\in(0,\pi).$ Then there exists $\rho\in(0,\frac{\pi}{2})$ such that
$t\in\lbrack\rho,\pi-\rho].$ Without loss of generality and for simplicity of
notation assume that $\lambda_{n}(t)$ for $n\geq N(\rho)$ are given. It
follows from (24), (25), (21) and (22) that
\[
\lambda_{n}(t)=(2\pi n+t)^{2}+O(\frac{1}{n}).
\]
Using it in (16) for $k=1$ we obtain
\[
a_{1}(\lambda_{n}(t))=\frac{ab}{(2\pi n+t)^{2}+O(n^{-1})-(2\pi(n-1)+t)^{2}}+
\]%
\[
\frac{ab}{(2\pi n+t)^{2}+O(n^{-1})-(2\pi(n+1)+t)^{2}}=
\]%
\[
\frac{ab}{2\pi(2\pi(2n-1)+2t)}-\frac{ab}{2\pi(2\pi(2n+1)+2t)}+O(\frac{1}%
{n^{3}})=\frac{ab}{2(2\pi n+t)^{2}}+O(\frac{1}{n^{3}}).
\]
This with (24), (25), (21) and (22) implies that
\[
\lambda_{n}(t)=(2\pi n+t)^{2}+\frac{ab}{2(2\pi n+t)^{2}}+O(\frac{1}{n^{3}}).
\]
From this we find $ab$ by calculating the limit: $\lim_{n\rightarrow\infty
}(\lambda_{n}(t)-(2\pi n+t)^{2})2(2\pi n+t)^{2}.$ If $t=0,$ then using the
well-known [4] asymptotic formula
\[
\lambda_{n}(t)=(2\pi n)^{2}+\frac{ab}{2(2\pi n)^{2}}+O(\frac{1}{n^{3}})
\]
in the same way we determine $ab.$ The case $t=\pi$ can be considered in the
same way.
\end{proof}

\begin{theorem}
The following conditions are equivalent

$1)$ $ab=cd$

$2)$ $S(H_{t}(a,b))=S(H_{t}(c,d))$ for all $t\in(-\pi,\pi]$

$3)$ $S(H_{t}(a,b))=S(H_{t}(c,d))$ for some $t\in(-\pi,\pi]$

$4)$ $\lambda_{n_{k}}(t,a,b)=\lambda_{n_{k}}(t,c,b)$ for some $t\in(-\pi,\pi]$
and for some sequence $\{n_{k}\}$.

$5)$ $S(H(a,b))=S(H(c,d)).$
\end{theorem}

\begin{proof}
By Theorem 1, $1)$ implies $2)$ and $5)$. It is clear that $2)\Longrightarrow
3)\Longrightarrow4).$ By Theorem 2, $4)\Longrightarrow1).$ Thus $1)$, $2),$
$3)$ and $4)$ are equivalent. It remains to show that $5)$ implies at least
one of them. By Theorem 5 of [10] there exists $N$ such that for $|n|>N$ the
component $\Gamma_{n}$ of the spectrum of the operator $H$ is separated simple
analytic arc with end points $\lambda_{n}(0)$ and $\lambda_{n}(\pi).$
Therefore $5)$ implies $4).$
\end{proof}

\ Now we obtain some consequences of Theorem 1. First consequence is the
generalization of formula (3.25) of [1], which is the extension of the
asymptotic formula of Harrell-Avron-Simon [2,7], for the case $a\neq b.$ For
simplicity of reading we write this result in notation of [1]. Let
$\lambda_{n}^{\pm}(a,b)$ be the eigenvalue of the Mathieu operator with
potential%
\begin{equation}
ae^{-i2x}+be^{i2x}%
\end{equation}
and with periodic (if $n$ is even) or antiperiodic (if $n$ is odd) boundary
condition. It is proved in [1] that if $b=a,$ then the following asymptotic
formula holds
\begin{equation}
\lambda_{n}^{+}-\lambda_{n}^{-}=\pm8a^{n}4^{-n}((n-1)!)^{-2}(1-\frac{a^{2}%
}{4n^{3}}+O(\frac{1}{n^{4}})).
\end{equation}
It is clear that Theorem 1 continues to hold for (34), since this case can be
reduced to the case (2) by substitution $s=\pi x$.

\begin{corollary}
For every complex numbers $a$ and $b$ the following formula holds
\begin{equation}
\lambda_{n}^{+}(a,b)-\lambda_{n}^{-}(ab)=\pm8(ab)^{\frac{n}{2}}4^{-n}%
((n-1)!)^{-2}(1-\frac{ab}{4n^{3}}+O(\frac{1}{n^{4}})).
\end{equation}

\end{corollary}

\begin{proof}
First proof of (36): Let $c=(ab)^{\frac{1}{2}}.$ By Theorem 1, $\lambda
_{n}^{\pm}(a,b)=\lambda_{n}^{\pm}(c,c).$ Therefore in (35) replacing $a^{2}$
with $ab\ $\ we obtain (36).

Second proof of (36): $\lambda_{n}^{\pm}(a,b)=n^{2}+z$ for large $n$ is an
eigenvalue of periodic or antiperiodic boundary condition if and only if $z$
is a root of the equation (18) of [3]:%
\[
(z-a(n,z))^{2}=B^{+}(n,z)B^{-}(n,z),
\]
where $a(n,z),B^{+}(n,z)$,$B^{-}(n,z)$ are defined as $A(\lambda
,t)),B(\lambda,t),B^{^{\prime}}(\lambda,t)$ for $t=0.$ Hence by Remark 1
$a(n,z)$ and $B^{+}(n,z)B^{-}(n,z)$ depend only on $ab.$ Therefore arguing as
above we obtain (36) from (35).
\end{proof}

Now we consider the operator $L_{t}(q)$ generated by (1) and (3) when%
\begin{equation}
q\in L_{1}[0,1],\text{ }q_{n}=0,\forall n=0,-1,-2,...
\end{equation}
Spectral theory of the operator $L(q)$ generated in $L_{2}(-\infty,\infty)$ by
the expression (1) with the potential $q$ satisfying (37) and the additional
condition $%
{\textstyle\sum\nolimits_{n}}
\mid q_{n}\mid<\infty$ is studied by Gasymov [5].

\begin{theorem}
$(a)$The eigenvalues of the operators $L_{t}(q)$ for $t\in(-\pi,\pi]$ with
potential (37) are $(2\pi n+t)^{2}$ ,where $n\in\mathbb{Z}$. These eigenvalues
for $t\neq0,\pi$ are simple. The eigenvalues $(2\pi n)^{2}$ for $n\in
\mathbb{Z}\backslash\{0\}$ and $(2\pi n+\pi)^{2}$ for $n\in\mathbb{Z}$ are
double eigenvalues of $L_{0}(q)$ and $L_{\pi}(q)$ respectively. The theorem
continues to hold if (37) is replaced by%
\begin{equation}
q\in L_{1}[0,1],\text{ }q_{n}=0,\forall n=0,1,2,...
\end{equation}

$(b)$ Let $\Psi_{n,t}(x)$ be the eigenfunction of the operator $L_{t}(q)$
corresponding to the eigenvalue $(2\pi n+t)^{2}$ and normalized as
\begin{equation}
(\Psi_{n,t},e^{i(2\pi n+t)x})=1,
\end{equation}
where $t\neq0,\pi$ and $q$ satisfies (37). Then\textit{ }%
\begin{equation}
\text{ }\Psi_{n,t}(x)=e^{i(2\pi n+t)x}+\sum_{p\in\mathbb{N}}c_{p}%
(t)e^{i(2\pi(n+p)+t)x}),
\end{equation}
where $c_{1}(t)=q_{1}d_{1}(t),$ $d_{p}(t)=((2\pi n+t)^{2}-(2\pi(n+p)+t)^{2}%
)^{-1},$
\[
c_{2}(t)=d_{2}(t)(q_{2}+\frac{q_{1}q_{1}}{(2\pi n+t)^{2}-(2\pi(n+1)+t)^{2}}),
\]%
\begin{equation}
c_{p}=d_{p}\left(  q_{p}+\sum_{k=1}^{p-1}\sum_{n_{1},n_{2},...,n_{k}}%
\frac{q_{n_{1}}q_{n_{2}}...q_{n_{k}}q_{p-n_{1}-n_{2}-...-n_{k}}}{%
{\textstyle\prod\limits_{s=1}^{k}}
[(2\pi n+t)^{2}-(2\pi(n+p-n_{1}-n_{2}-...-n_{s})+t)^{2}]}\right)
\end{equation}
for $p=3,4,...$ and
\begin{equation}
\{n_{1},n_{2},...,n_{s},p-n_{1}-n_{2}-...-n_{s}\}\subset\mathbb{N}%
=:\{1,2,...,\}
\end{equation}
for $s=1,2,...,p-1.$ The theorem continues to hold if (37) is replaced by (38)
and $\mathbb{N}$ in  formulas (40) and (42) is replaced by $-\mathbb{N}%
$\textit{.}
\end{theorem}

\begin{proof}
$(a)$ Formulas (15)-(17) for $q\in L_{1}[0,1]$ with estimations that
guaranties the equality $\lim_{m\rightarrow\infty}R_{m+1}=0$ are proved in
[9]. Since at least one of the indices $n_{1},n_{2},...,n_{k},$ $-n_{1}%
-n_{2}-...-n_{k}$ is not positive number, by (16) and (37), $a_{k}(\lambda
_{n}(t))=0$ and hence $A_{m}(\lambda_{n}(t))=0.$ Therefore letting $m$ tend to
infinity and using (8) and (15) we obtain $\lambda_{n}(t)=(2\pi n+t)^{2}$ for
$t\in\lbrack\rho,\pi-\rho]$ and $n\gg1.$ Now arguing as in the proof of
Theorem 1 we get the proof of this theorem in case (37). Case (38) can be
proved in the same way.

$(b)$ Let $\Psi_{n,t}(x)$ be the normalized eigenfunction of the operator
$L_{t}(q)$ corresponding to the eigenvalue $(2\pi n+t)^{2}$ and $t\neq0,\pi.$
(In the end we prove that there exists an eigenfunction of the operator
$L_{t}(q)$ satisfying (39). For simplicity of notation we denote it also by
$\Psi_{n,t}.$) Since $\{e^{i(2\pi(n+p)+t)x}:p\in\mathbb{Z}\}$ is an
orthonormal basis we have
\begin{equation}
\Psi_{n,t}(x)-(\Psi_{n,t},e^{i(2\pi n+t)x})e^{i(2\pi n+t)x}=\sum
_{p\in\mathbb{Z}\backslash\{0\}}(\Psi_{n,j},e^{i(2\pi(n+p)+t)x})e^{i(2\pi
(n+p)+t)x}.
\end{equation}
To find $(\Psi_{n,j},e^{i(2\pi(n+p)+t)x})$ we iterate (6) (in (6) replace $-k$
with $p)$ by using
\begin{equation}
(q\Psi_{n,t},e^{i(2\pi(n+p)+t)x})=\sum_{n_{1}}q_{n_{1}}(\Psi_{n,t}%
(x),e^{i(2\pi(n+p-n_{1})+t)x})
\end{equation}
(see (14) of [9]) and the equality $\lambda_{n}(t)=(2\pi n+t)^{2}$ (see
$(a)$). Namely, (6) with (44) implies
\begin{equation}
(\Psi_{n,t},e^{i(2\pi(n+p)+t)x})=d_{p}(t)\sum_{n_{1}}q_{n_{1}}(\Psi
_{n,t},e^{i(2\pi(n+p-n_{1})+t)x}).
\end{equation}
Now arguing as in the proof of (15), that is, isolating the terms in the
right-hand side of (45) containing the multiplicand $(\Psi_{n,t},e^{i(2\pi
n+t)x})$ which occurs in the case $n_{1}=p$ and using again (45) for the other
terms and etc., we get%
\begin{equation}
(\Psi_{n,t},e^{i(2\pi(n+p)+t)x})=c_{p}(t)(\Psi_{n,t},e^{i(2\pi n+t)x})+r_{m},
\end{equation}
where
\begin{equation}
r_{m}=\sum_{n_{1},n_{2},...,n_{m}}\frac{d_{p}(t)q_{n_{1}}q_{n_{2}}...q_{n_{m}%
}(q\Psi_{n,t},e^{i(2\pi(n+p-n_{1}-n_{2}-...-n_{m})+t)x})}{%
{\textstyle\prod\limits_{s=1,2,...,m}}
[(2\pi n+t)^{2}-(2\pi(n+p-n_{1}-n_{2}-...-n_{s})+t)^{2}]},
\end{equation}
$p-n_{1}-n_{2}-...-n_{s}\neq0$ for $s=1,2,...,m$ and $m>p.$ The indices
$n_{1},n_{2},...,n_{p-1}$ taking part in the expression of $c_{p}(t)$ satisfy
(42). Therefore if $p<0,$ then the set of these indices is empty, that is,
$c_{p}(t)=0.$ Moreover if $p>0$ then, by (42), the number of summands of
$c_{p}(t)$ is finite.

Now we prove that $r_{m}\rightarrow0$ as $m\rightarrow\infty.$ Let $M=\sup
_{n}\mid q_{n}\mid.$ By (37) $n_{k}\geq1$ for  $k=1,2,...,m$ and hence
$n_{1}+n_{2}+...+n_{s}\geq s.$ Using this and taking into account that
\ $(q\Psi_{n,t},e^{i(2\pi(n+p-n_{1}-n_{2}-...-n_{m})+t)x})\rightarrow0$ as
$m\rightarrow\infty$ we obtain
\begin{equation}
\mid r_{m}\mid\leq\mid d_{p}(t)\mid%
{\textstyle\prod\limits_{s=1,2,...,m}}
\left(  \sum_{j\geq s,\text{ }j\neq p}\frac{M}{\mid(2\pi n+t)^{2}%
-(2\pi(n+p-j+t)^{2}\mid}\right)
\end{equation}
for $m\gg1$. Clearly there exist $K(t)$ such that
\begin{equation}
\sum_{j\geq s,\text{ }j\neq p}\mid\frac{M}{[(2\pi n+t)^{2}-(2\pi
(n+p-j+t)^{2}]}\mid\leq K(t).
\end{equation}
for $s=1,2,...,m.$ Moreover if $s\geq4(\mid n\mid+\mid p\mid)$ then%
\begin{equation}
\sum_{j\geq s}\mid\frac{M}{[(2\pi n+t)^{2}-(2\pi(n+p-j+t)^{2}]}\mid
<\sum_{j\geq s}\frac{M}{j^{2}}<\frac{M}{s}%
\end{equation}
Now using (48)-(50) we obtain
\begin{equation}
\mid r_{m}\mid\leq\frac{\mid d_{p}(t)\mid M^{m-4(\mid n\mid+\mid p\mid
)+1}(K(t))^{4(\mid n\mid+\mid p\mid)-1}}{4(\mid n\mid+\mid p\mid)(4(\mid
n\mid+\mid p\mid)+1)...m}%
\end{equation}
which implies that $r_{m}\rightarrow0$ as $m\rightarrow\infty.$ Therefore in
(46) letting $m$ tend to infinity we get
\begin{equation}
(\Psi_{n,t},e^{i(2\pi(n+p)+t)x})=c_{p}(t)(\Psi_{n,t},e^{i(2\pi n+t)x}).
\end{equation}
This with (43) shows that $(\Psi_{n,t},e^{i(2\pi n+t)x})\neq0.$ Therefore,
there exists eigenfunction, denoted again by $\Psi_{n,t},$ satisfying (39) and
for this eigenfunction, by (52), we have
\[
(\Psi_{n,t},e^{i(2\pi(n+p)+t)x})=c_{p}(t).
\]
Thus (40) is proved in case (37). The case (38) can be considered in the same way.
\end{proof}

\end{document}